\documentclass{amsart}
\usepackage{graphicx,amsmath,amsthm,amssymb}

\newtheorem{proposition}{Proposition}[section]
\newtheorem{definition}{Definition}[section]
\newtheorem{lemma}{Lemma}[section]
\newtheorem{theorem}{Theorem}[section]
\theoremstyle{definition}
\newtheorem{remark}{Remark}[section]

\DeclareMathOperator{\vol}{vol}
\DeclareMathOperator{\diag}{diag}

\title[A Cheeger Inequality based on a reflection principle]{A Cheeger
inequality for graphs based on a \\ reflection principle}

\author[Gelernt]{Edward Gelernt}
\address{Department of Mathematics, Yale University, New Haven,
CT 06511, USA} 
\author[Halikias]{Diana Halikias}
\address{Department of Mathematics, Yale University, New Haven,
CT 06511, USA} 
\author[Kenney]{Charles Kenney}
\address{Department of Mathematics, Rutgers University,
Piscataway, NJ 08854, USA}
\author[Marshall]{Nicholas F.  Marshall}
\address{Department of Mathematics, Princeton University,
Princeton, NJ 08542, USA}

\keywords{Cheeger inequality, graph Laplacian, Neumann Laplacian}
\subjclass[2010]{05C50, 05C85 (primary) and 15A42 (secondary) }

\begin{document}

\begin{abstract}
Given a graph with a designated set of boundary vertices, we define a new notion
of a Neumann Laplace operator on a graph using a reflection principle.  We show
that the first eigenvalue of this Neumann graph Laplacian satisfies a Cheeger
inequality.
\end{abstract}

\maketitle

\section{Introduction and Main Result}
\subsection{Introduction} \label{intro}
Suppose that $G = (V,E)$ is a graph with vertices $V$ and edges $E$. 
Let $\partial V \subseteq V$ be a designated set of boundary vertices, and
$\mathring{V} := V \setminus \partial V$.
We define the doubled graph $G'$ as
follows. Let $\mathring{G} =(U,F)$ be an isomorphic copy of the induced subgraph
$G[\mathring{V}]$, and let $f$ be an isomorphism from $\mathring{V}$ to $U$. Set
$$
F' := \left\{ \{u,v\} : u \in U, v \in \partial V, \{f^{-1}(u),v\}
\in E \right\}.
$$
Then, we define $G' := (V',E')$ where $V' := V \cup U$ and $E' := E \cup F \cup
F'$. That is to say, $G'$ is defined by making an isomorphic copy of the
interior of $G$ and attaching it to the boundary vertices $\partial V$ as in the
original graph, see Figure \ref{fig01}.

\begin{figure}[h!]
\centering
\includegraphics[width= .8\textwidth]{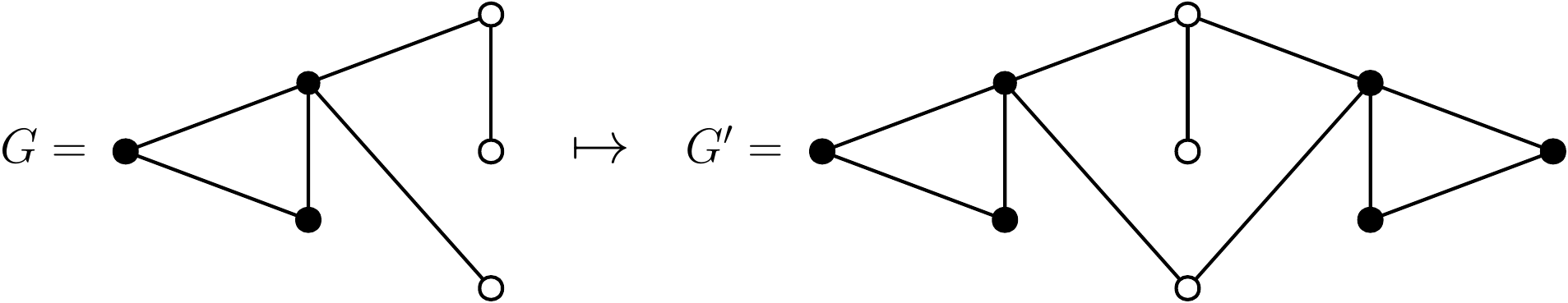}
\caption{A graph $G$, and its doubled graph $G'$, where the black and white dots
denote interior and boundary vertices, respectively.}
\label{fig01}
\end{figure}

\begin{definition}
Let $G'=(V',E')$ be a doubled graph, and let $f:\mathring{V} \to U$ be an isomorphism as above, 
so that for all $w \in \partial V$
and $v \in \mathring{V}, \ \{v,w\} \in E' \iff \{f(v),w\} \in E'.$
We say that a function $\varphi : V' \rightarrow \mathbb{R}$ is even with
respect to $\partial V$ if
$$
\varphi(v) = \varphi(f(v)) \text{ for } v \in \mathring{V},
$$ 
and we say that $\varphi$ is odd with respect to $\partial V$ if
$$
\varphi(v) = -\varphi(f(v)) \text{ for } v \in \mathring{V}, \quad \text{and} \quad
\varphi(v) = 0 \text{ for } v \in \partial V.
$$ 
\end{definition}

Let $L' := D - A$ denote the graph Laplacian of $G'$ where $D$ is the degree
matrix of $G'$, and $A$ is the adjacency matrix of $G'$. The following
proposition characterizes the eigenvectors of $L'$ as
either even or odd.

\begin{proposition} \label{prop1}
The graph Laplacian $L'$ has $|V|$ eigenvectors that are even with respect to
$\partial V$, and $|\mathring{V}|$ eigenvectors that are odd with respect to
$\partial V$; this accounts for all eigenvectors of $L'$.
\end{proposition}

\subsection{Motivation}
We are motivated by the observation that the restriction of the odd and even
eigenvectors of $L'$ to the graph $G$ seem like natural Dirichlet and
Neumann Laplacian eigenvectors for the graph $G$, given the respective odd and
even behavior of Dirichlet and Neumann Laplacian eigenfunctions on manifolds. In
fact, the restriction of the odd eigenvectors of $L'$ to the graph $G$ are
eigenvectors of the Dirichlet graph Laplacian defined by Chung in
\cite{Chung1997}, and inequalities involving the eigenvalues of this operator
have been investigated \cite{ChungOden2000}.  However, an operator corresponding
to the restriction of the even eigenvectors of $L'$ to $G$ has not, to our
knowledge been investigated.  In \cite{Chung1997}, Chung defines the Neumann
graph Laplacian by enforcing a condition that a discrete derivative vanishes on
the boundary nodes of the graph, which results in different eigenvectors than
those arising from the even eigenvectors of $L'$. We note that a Cheeger
inequality for Chung's definition of the Neumann graph Laplacian has recently
been established by Hua and Huang \cite{HuaHuang2018}. 

\subsection{Odd and even eigenvectors}
The proof of Proposition \ref{prop1} gives some initial insight into the odd
and even eigenvectors the graph Laplacian $L'$ on the doubled graph $G'$.

\begin{proof}[Proof of Proposition \ref{prop1}]
The proof of this proposition is immediate from the block structure of the graph
Laplacian $L'$. Indeed, let $L'(U,W)$ denote the submatrix of $L'$ whose rows
and columns are indexed by $U \subseteq V$ and $W \subseteq V$, respectively. We
can write
$$
L' =
\left(\begin{array}{ccc}
X & Y & 0 \\
Y^\top & Z & Y^\top \\
0 & Y & X 
 \end{array}\right),
$$
where $X$ is the submatrix $L'(\mathring{V},\mathring{V})$, $Y$ is the submatrix
$L'(\mathring{V},\partial V)$, and $Z$ is the submatrix $L'(\partial V,\partial V)$.
With this notation, the eigenvectors of $L'$ that are even with respect to
$\partial V$ are solutions to the equation
$$
\left(\begin{array}{ccc}
X & Y & 0 \\
Y^\top & Z & Y^\top \\
0 & Y & X 
 \end{array}\right)
\left( \begin{array}{c}
u\\
v\\
u
\end{array} \right)
 =
\mu
\left(
\begin{array}{c}
u\\
v\\
u
\end{array} \right).
$$
That is to say, the vectors $u$ and $v$ satisfy $Xu + Yv = \mu u$ and
$2Y^\top u + Zv = \mu v$. Put differently, when concatenated, $u$ and $v$ form
an eigenvector of the matrix
\begin{equation} \label{defLR}
L_R :=
\begin{pmatrix}
X & Y\\
2Y^\top & Z
\end{pmatrix}.
\end{equation}
Observe that $L_R$ is similar to a symmetric matrix
$$
L_R = \begin{pmatrix} I & 0\\ 0 & \sqrt{2}I \end{pmatrix}
\begin{pmatrix} X & \sqrt{2}Y\\
\sqrt{2} Y^\top & Z \end{pmatrix}
\begin{pmatrix} I & 0\\ 0 & \sqrt{2}I \end{pmatrix}^{-1},
$$
and thus by the Spectral Theorem, $L_R$ has $|V|$ real eigenvectors, which give
rise to $|V|$ even eigenvectors of $L'$.  The eigenvectors of $L'$ that are
odd with respect to $\partial V$ are solutions to the equation $$
\left(\begin{array}{ccc}
X & Y & 0 \\
Y^\top & Z & Y^\top \\
0 & Y& X 
 \end{array}\right)
\left(
\begin{array}{c}
u\\
0\\
-u
\end{array} \right)
 =
\lambda
\left( \begin{array}{c}
u \\
0\\
-u
\end{array} \right).
$$
Thus, each vector $u$ such that $X u = \lambda u$ gives rise to an odd
eigenvector of $L'$. Let
$$
L_D := X.
$$
Since $L_D$ is symmetric, it follows from the Spectral Theorem that it has
$|\mathring{V}|$ real eigenvectors, and we conclude that $L'$ has
$|\mathring{V}|$ odd eigenvectors.  \end{proof}

\subsection{Contribution}
In this paper, we study the operator $L_R$ defined in \eqref{defLR} which we
call the reflected Neumann graph Laplacian. This operator seems to be
particularly natural on graphs approximating manifolds. For example, in Remark
\ref{rmk1}, we show that on the path graph, the eigenvectors of the Dirichlet
graph Laplacian $L_D$ and reflected Neumann graph Laplacian $L_R$ are the
familiar discrete sine and cosine functions. We remark that the definition of
the reflected Neumann graph Laplacian $L_R$ has some similarities to the
normalization used in the diffusion maps manifold learning method of Coifman and
Lafon \cite{CoifmanLafon2006}.

Our main result Theorem \ref{thm1} shows that the first
eigenvalue of the normalized reflected Neumann graph Laplacian $\mathcal{L}_R$
defined in \eqref{nlr} satisfies a Cheeger inequality.  
The graph cuts arising from $\mathcal{L}_R$ can differ significantly from graph
cuts arising from the standard normalized graph Laplacian $\mathcal{L}$ defined
in \cite{Chung1997}. In Figure \ref{fig04}, we illustrate Theorem \ref{thm1}
with an example where the first eigenvector of the Neumann graph Laplacian $L_R$
suggests a drastically different cut than the first eigenvector of the standard
graph Laplacian, and describe how the graph cut suggested by $L_R$ is consistent
with the Cheeger inequality established in Theorem \ref{thm1}. It may be
interesting to investigate the analog of other classical eigenvalue inequalities
involving these definitions of $L_D$ and $L_R$ for graphs with boundary.

\begin{remark} \label{rmk1}
The operators $L_D$ and $L_R$ are particularly natural on the path graph.
Let $P_n =(V,E)$ denote the path graph on $n$ vertices, where $V = \{1, \ldots,
n\}$ and $\{u,v\} \in E$  if and only if $|u - v| = 1$. If $\partial V :=
\{1,n\}$, then the doubled graph $P_n' = C_{2n-2}$ is the cycle graph on $2n -2$ vertices,
see Figure \ref{fig02}.

\begin{figure}[h!]
\centering
\begin{tabular}{@{}p{.25\textwidth}p{.15\textwidth}p{.25\textwidth}}
\raisebox{-.5\height}{\includegraphics[width=
.27\textwidth]{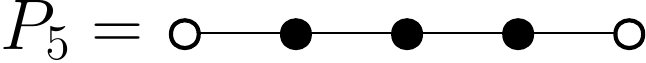}} &
$\qquad \mapsto$ &
\raisebox{-.5\height}{\includegraphics[width=
.27\textwidth]{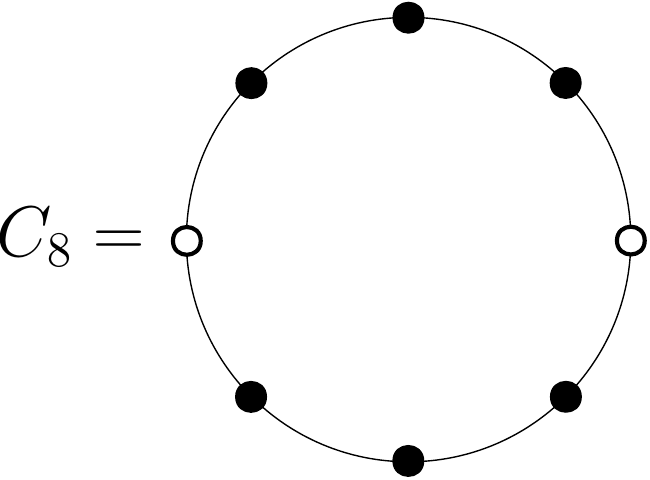}}
\end{tabular}
\caption{A path graph and its doubled graph.} \label{fig02}
\end{figure}

Consider $L_D$ and $L_R$ of the path
graph $P_n$.  The Dirichlet 
eigenvectors $\varphi_k$ and eigenvalues $\lambda_k$, which satisfy $L_D
\varphi_k = \lambda_k \varphi_k$ for $k = 1,\ldots,n-2$, are of the form 
$$
\lambda_k := 2 \left(1-\cos \left(\frac{\pi k}{n-1} \right) \right) \quad
\text{and} \quad
\varphi_k(j) = \sin \left(\frac{\pi j k}{n-1} \right),
$$
for $j=1,\ldots,n-2$, while the Neumann 
eigenvectors, $\psi_k$ and $\mu_k$, which satisfy $L_R \psi_k = \mu_k \psi_k$
for $k=0,\ldots,n-1$, are of the form
$$
\mu_k := 2 \left(1-\cos \left(\frac{\pi k}{n-1} \right) \right) \quad
\text{and} \quad
\psi_k(j) := \cos \left(\frac{\pi j k}{n-1} \right)
$$
for $j = 0,\ldots,n-1$. Thus, the path graph  doubling procedure defined in
\S \ref{intro} gives the familiar sine and cosine functions, which are the
Dirichlet and Neumann eigenfunctions of the Laplace operator of the unit
interval.
\end{remark}

\subsection{Notation and definitions}

Suppose that $G = (V,E)$ is a graph with vertices $V$ and edges $E$. Let
$\partial V \subseteq V$ be a designated set of boundary vertices, and set
$\mathring{V} = V \setminus \partial V$. We can write the adjacency matrix $A$ of the
graph $G$ as the block matrix
$$
A = \left( \begin{array}{cc}
A_{11} & A_{12} \\
A_{12}^\top & A_{22} 
\end{array}
\right),
$$
where $A_{11} = A(\mathring{V},\mathring{V})$, $A_{12} = A(\mathring{V},\partial
V)$, and $A_{22} = A(\partial V,\partial V)$. Motivated by Proposition
\ref{prop1} we define the reflected adjacency matrix $R$ by
$$
R := \left( \begin{array}{cc} A_{11} & A_{12} \\
2 A_{12}^\top & A_{22} 
\end{array}
\right).
$$
With this notation, the reflected Neumann Laplacian $L_R$ can be defined by
$$
L_R = D - R,
$$
where $D = \diag(R \vec{1})$, where $\vec{1}$ denotes a vector whose entries are
all $1$, and whose dimensions are such that the matrix-vector multiplication is
well defined. We define the normalized reflected
Neumann graph Laplacian $\mathcal{L}_R$ by
\begin{equation} \label{nlr}
\mathcal{L}_R := D^{-1/2} L_R D^{-1/2}.
\end{equation}

\subsection{Main result} \label{mainresult}
In this section, we present our main result Theorem \ref{thm1}.  While the
matrix $\mathcal{L}_R$ is not in general symmetric, it is similar to a symmetric
matrix; indeed, if 
$$
Q: = \left( \begin{array}{cc}
I_{|\mathring{V}|} &  0 \\
0 & \frac{1}{2} I_{|\partial V|}
\end{array}  \right),
$$
then $Q^{1/2} \mathcal{L}_R Q^{-1/2}$ is symmetric, positive-definite, and has
the eigenvector $D^{1/2} Q^{1/2} \vec{1}$ of eigenvalue $0$. It follows that the
first nontrivial eigenvalue $\lambda_R$ of $\mathcal{L}_R$ satisfies
$$
\lambda_R := \inf_{x^\top D^{1/2} Q^{1/2} \vec{1} = 0} \frac{x^\top Q^{1/2}
\mathcal{L}_R Q^{-1/2} x}{x^\top x}.
$$
Let $E(U,W) := \left\{\{u,w\} \in E : u \in U, w \in W \right\}$, that is,
$E(U,W)$ is the set of edges between $U$ and $W$. We define a measure $m(U,W)$
on this set of edges by
$$
m(U,W) = |E(U,W)| - \frac{1}{2} |E(U \cap \partial V,
W \cap \partial V)|,
$$
and we define the  volume $\vol(U)$ of $U \subseteq V$ by
$$
\vol(U) := \sum_{u \in U} m(\{u\},V).
$$
 The following theorem is our main result.
\begin{theorem} \label{thm1}
Suppose that $G = (V,E)$ is a graph with a designated set of boundary vertices
$\partial V \subseteq V$, and define the Cheeger constant $h_R$ by
\begin{equation} \label{hr}
h_R :=  \min_{S \subseteq V} \frac{m(S,V \setminus S)}{\min\{
\vol(S), \vol(V \setminus S) \}}.
\end{equation}
Then,
$$
\sqrt{2 \lambda_R} \ge h_R \ge \lambda_R/2,
$$
where $\lambda_R$ is the first nontrivial eigenvalue of $\mathcal{L}_R$.
\end{theorem}

Recall that the standard Cheeger inequality is constructive in the
sense that a cut that achieves the upper bound on the Cheeger constant can be
determined from the eigenfunction corresponding to the first eigenvalue of the
normalized graph Laplacian $\mathcal{L}$, see \cite{Alon1986,Cheeger1970}.
Specifically, a partition that achieves the upper bound can be determined by
dividing the vertices into two groups based on if the value of the first
eigenvector is more or less than some threshold; for a detailed exposition see
for example \cite{Chung1997,Chung2007}.  Similarly, the result of Theorem
\ref{thm1} is constructive in the sense that a cut which achieves the upper
bound on $h_R$ can be determined from the eigenvector $\psi_R$ of
$\mathcal{L}_R$ that corresponds to $\lambda_R$. In the following remark, we
present an example where the cut arising from $\psi_R$ differs significantly
from the cut arising from the first eigenvector $\psi$ of the standard
normalized graph Laplacian $\mathcal{L}$.

\begin{remark} 
Graph cuts arising from $\psi_R$ can differ significantly from graph
cuts arising from $\psi$. Indeed, on the left of Figure \ref{fig04} we
illustrate a graph whose vertices are colored by greyscale values  proportional
to $\psi$. On the right of Figure \ref{fig04} we illustrate the same graph
except several vertices have been designated as boundary vertices (indicated by
squares) and the color of the vertices is proportional to $\psi_R$. Observe that
$\psi$ suggests cutting the graph by a vertical line into two equal parts, while
$\psi_R$ suggests cutting the graph by a horizontal line into two equal parts.
\begin{figure}[h!] \centering
\begin{tabular}{cc}
\includegraphics[height=.3\textwidth]{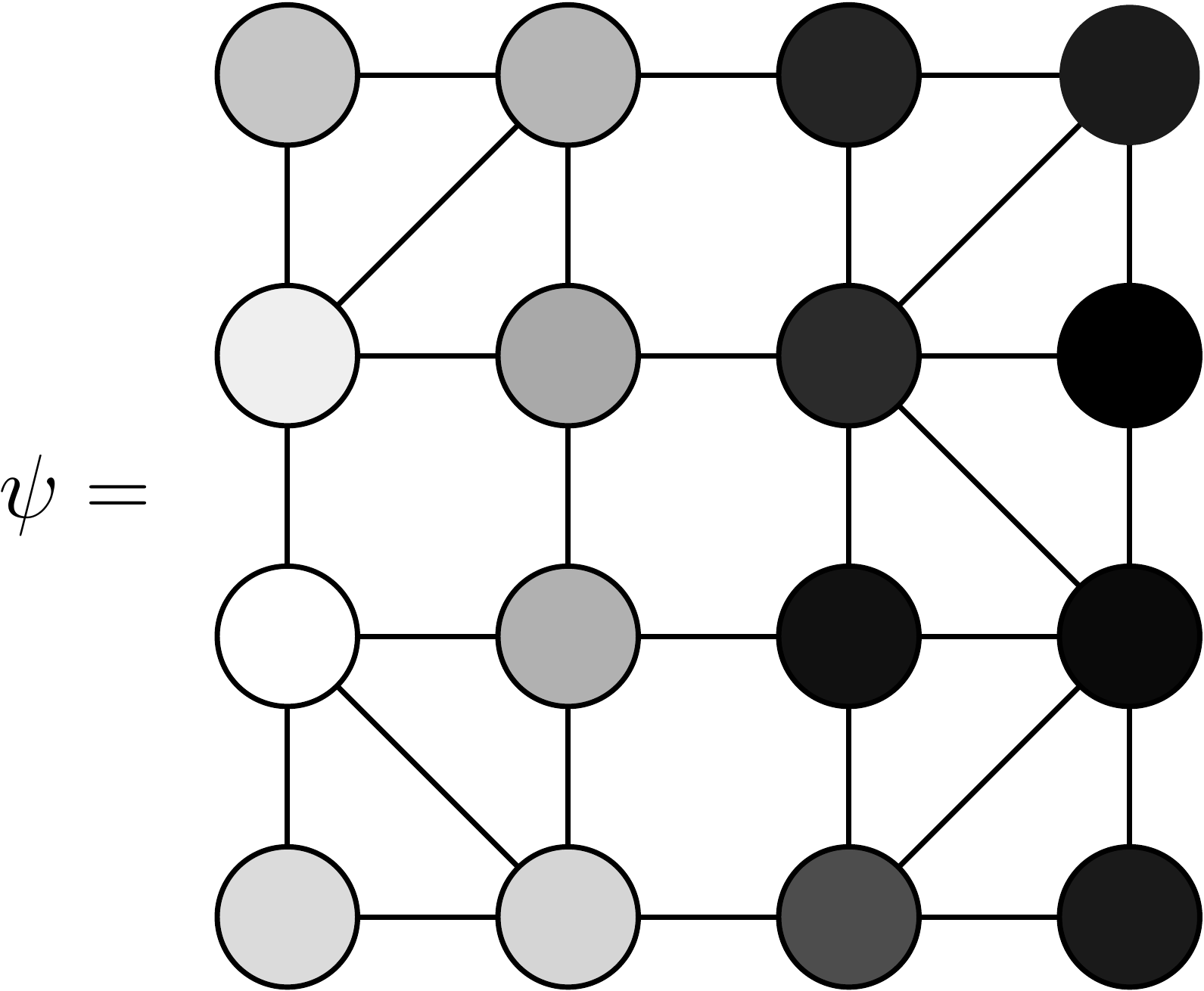} 
\includegraphics[height=.3\textwidth]{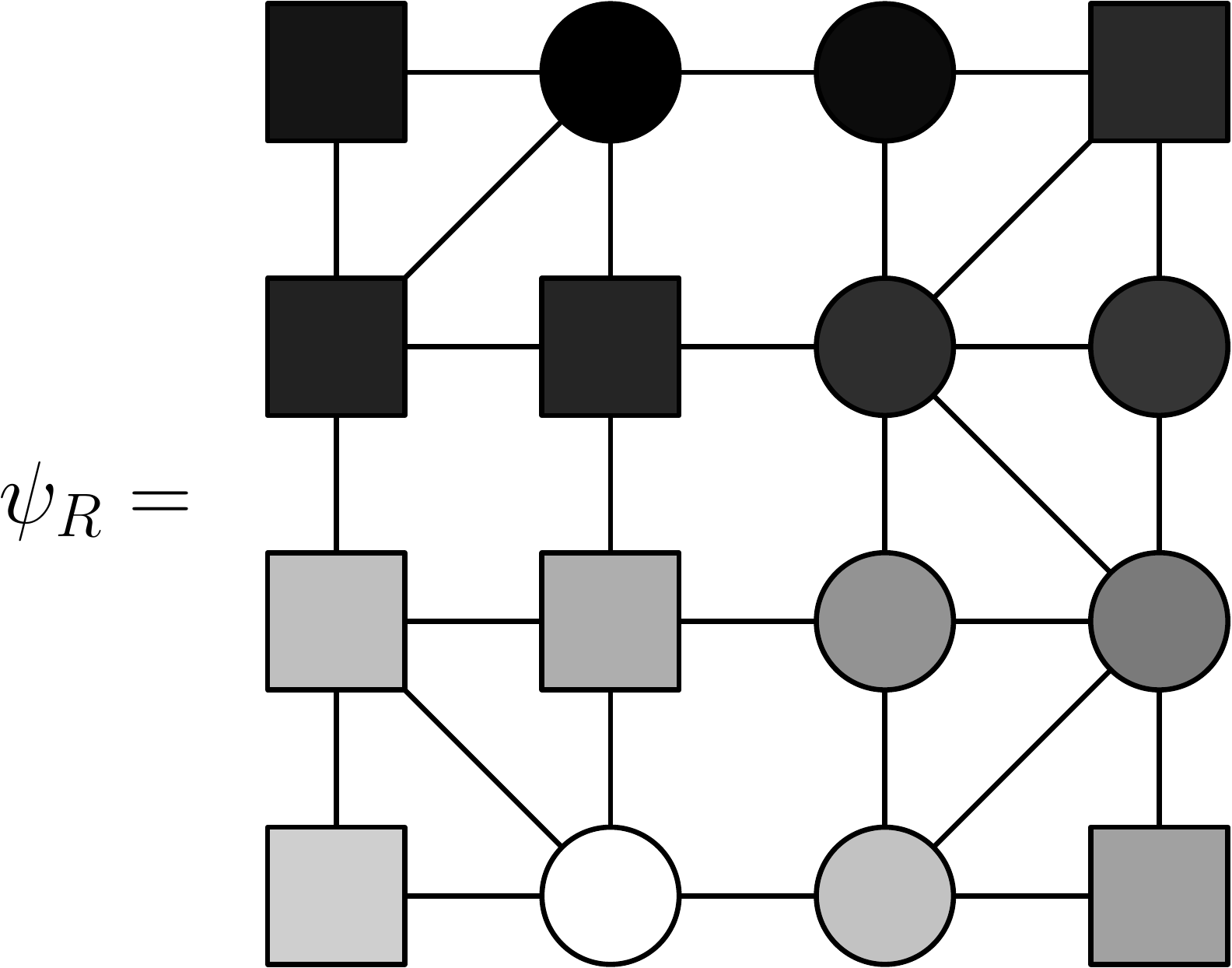} \quad &  \quad
\end{tabular}
\caption{The same graph with vertices colored proportional to $\psi$ (left)
and colored proportional to $\psi_R$ (right), where the squares in the right
graph denote boundary vertices.}
\label{fig04}
\end{figure}

The fact that $\psi_R$ suggests a horizontal cut of the graph is illustrative of
Theorem \ref{thm1}. Indeed, it is straightforward to check that the horizontal
cut suggested by $\psi_R$ minimizes the cut measure $m(S,V \setminus S)/(
\vol(S), \vol(V \setminus S) \})$ from \eqref{hr}. In contrast, the vertical cut
suggested by $\psi$ minimizes the standard cut measure, which is equivalent to
the measure $m(S,V \setminus S)/( \vol(S), \vol(V \setminus S) \})$ in the case
that all vertices are interior vertices.  Of course, Theorem \ref{thm1} only
guarantees that the measure of the cut arising from the eigenvector $\psi_R$  is
an upper bound for $h_R$ with value at most $\sqrt{2 \lambda_R}$; however, in
this simple example the cut arising from $\psi_R$ actually obtains this minimum.
\end{remark}

\begin{remark}
Here we visualize the first eigenfunction $\psi_R$ of the reflected Neumann
graph Laplacian $\mathcal{L}_R$ on a classic barbell shaped graph, see
Figure \ref{fig05}.
\begin{figure}
\centering
\includegraphics[width=.5\textwidth]{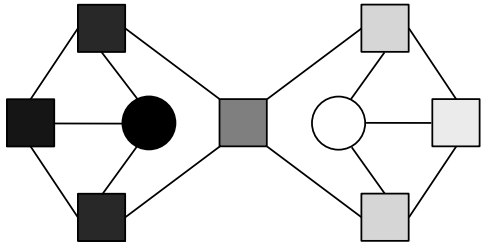} 
\caption{A barbell shaped graph whose vertices are colored proportional to
$\psi_R$, where squares in the graph denote boundary vertices.}
\label{fig05}
\end{figure}
Observe that in Figure \ref{fig05} the maximum and minimum value of the
eigenvector occur at an interior vertex. This feature of the eigenvectors is
interesting in the context of spectral clustering, where extreme values of the
eigenvectors often correspond to the center of clusters.

\end{remark}

\subsection{Future Directions}
One future direction for this work is the problem of selecting boundary vertices
in a principled way.  How the boundary is selected may depend on the application
at hand. In a social network graph, boundary vertices could correspond to
individuals with many connections outside the network.  In the context of
manifold learning, where the vertices of the graph are points in $\mathbb{R}^n$,
boundary vertices could be selected based on the number of points within some
$\varepsilon$-neighborhood of each vertex. On the other hand,
when a graph is given by sampling from a pre-defined manifold with boundary,
vertices selected from some collar neighborhood of the boundary could be
designated as boundary vertices.

Another future direction arises from generalizing the setup under which our work
was done.  Our graph doubling procedure inputs a graph with boundary and outputs
a larger graph, containing the original graph as an induced subgraph, which has
a special $Z_2$ symmetry. Could similar Cheeger results be proven for other
reflection procedures? For example, what if $n-1$ copies of the interior
vertices were attached, instead of only 1?

Finally, we note a connection between the doubled graph (defined in \S
\ref{intro}) and numerical analysis, that may motivate a direction for future
study. Recall that for a path graph $P_n$ the eigenfunctions of the reflected
Neumann Laplacian $L_R$ are of the form $\psi_k(j) = \cos(\pi j k/(n-1))$, see
Remark \ref{rmk1}. These Neumann eigenvectors are precisely the basis vectors
for the Discrete Cosine Transform (DCT) Type I, as classified in
\cite{Strang1999}. The DCT Type II, which has basis vectors $\psi_k(j) = \cos(
\pi (j+1/2) k/n)$ is also important in numerical analysis; it could be
interesting to develop a graph doubling procedure whose Neumann eigenvectors on
the path graph are these vectors.

\section{Proof of Main Result}
\subsection{Summary}
The proof of Theorem
\ref{thm1} is divided into two lemmas: first, in Lemma \ref{lem2} we show that
$\lambda_R \le 2 h_R$, and  second, in Lemma \ref{lem3} we show that $h_R^2/2 \le
\lambda_R$. The structure of our argument is similar to classical
Cheeger inequality proofs, see \cite{Chung1997,Chung1996}.
\subsection{Proof of Theorem \ref{thm1}}
\begin{lemma}[Trivial direction] \label{lem2}
We have 
$$
\lambda_R \le 2 h_R.
$$
\end{lemma}
\begin{proof}[Proof of Lemma \ref{lem2}]
Recall that 
$$
\mathcal{L}_R := D^{-1/2} Q^{1/2} L_R Q^{-1/2} D^{-1/2}.
$$
First, we observe that $Q L_R$ can be written as 
$$
Q L_R = L - \frac{1}{2} L_\partial,
$$
where
$$
L = 
 \left( \begin{array}{cc}
\diag( A_{11} \vec{1} + A_{12} \vec{1})
- A_{11} & -A_{12} \\ 
-A_{12}^\top  & \diag( A_{12}^\top \vec{1} + A_{22} \vec{1} )
 - A_{22} 
\end{array} \right),
$$
and
$$
L_\partial := \left( \begin{array}{cc}
0 & 0 \\
0  & \diag(A_{22}\vec{1}) - A_{22}
\end{array} \right).
$$
Observe that $L$ is the standard graph Laplacian of $G$, while $L_\partial$ is the graph
Laplacian of the vertex induced subgraph $G[\partial V]$.  Fix a subset $S
\subseteq V$, and let $\chi_S$ be the indicator function for $S$. Define
$$
x := Q^{1/2} D^{1/2} \chi_S - \frac{\chi_S^\top D Q \vec{1}}{\vec{1}^\top D Q \vec{1}} D^{1/2}
Q^{1/2} \vec{1}.
$$
By construction,  we have $x^\top D^{1/2} Q^{1/2} \vec{1} = 0$, and it follows that
\begin{eqnarray*}
\lambda_{N} &\le&  \frac{x^\top D^{-1/2} Q^{1/2} L_R Q^{-1/2} D^{-1/2} x}{x^\top
x} \\ 
&=& \frac{\chi_S^\top Q L_R \chi_S}{ \chi_S^\top D Q \chi_S \left(
\vec{1} - \frac{\chi_S^\top D Q \chi_S}{\vec{1}^\top D Q 1}\right)} \\
&=& \frac{\chi_S^\top (L - \frac{1}{2} L_\partial) \chi_S \left( \vec{1}^\top D Q \vec{1}
\right)}{ \left( \chi_S^\top D Q \chi_S^\top\right) \left( \chi_{V \setminus
S}^\top D Q \chi_{V \setminus S} \right)} \\
&\le& \frac{2 \cdot \chi_S^\top (L - \frac{1}{2} L_\partial) \chi_S }{\min \left\{ \left(
\chi_S^\top D Q \chi_S^\top\right), \left( \chi_{V \setminus S}^\top D Q
\chi_{V \setminus S} \right) \right\}}  \\
&=& \frac{2 \cdot m(S,V \setminus S)}{\min\{
\vol(S), \vol(V \setminus S) \}}.
\end{eqnarray*}
Since this inequality holds for all subsets $S \subseteq V$, we conclude that
$\lambda_R \le 2 h_R$, as was to be shown. 
\end{proof}

\begin{lemma}[Nontrivial direction] \label{lem3}
We have
$$
\lambda_R \ge \frac{h_R^2}{2}.
$$
\end{lemma}
\begin{proof}[Proof of Lemma \ref{lem3}]
Recall that 
$$
\lambda_R = \inf_{x^\top D^{1/2} Q^{1/2} \vec{1} = 0} \frac{x^\top
\mathcal{L}_R
x}{x^\top x} = \inf_{y^\top D Q \vec{1} = 0} \frac{y^\top Q L_R y}{y^\top Q D y}.
$$
Let $g$ be a vector satisfying
$$
\lambda_R = \frac{g^\top Q L_R g}{g^\top D Q g}, \quad \text{and} \quad
g^\top Q D \vec{1} = 0.
$$
Let $\{v_1,\ldots,v_n\} $ be an enumeration of the vertices $V$ so that
$g_{v_1} \le  ...\le g_{v_n}$, and set $S_j := \{v_1, ..., v_j\}$, for $j =
1,\ldots,n.$ Let $p$ be the largest integer such that $\vol(S_p) \le \vol(V)/2$,
that is,
$$
p := \max \left\{ j \in \{1,\ldots,n\} :\text{vol}(S_j) \leq \text{vol}(V)/2
\right\}.
$$
Let $g^+$ and $g^-$ denote the positive and negative parts of $g - g_{v_p}$,
respectively. That is,
$ g^+_v := \max \{g_v-g_{v_p}, 0 \}$
and
$g^-_v := \max \{g_{v_p}-g_v, 0\}.$
Let $u \sim v$ denote $\{u,v\} \in E$ and $q = \diag(Q).$ Then
\begin{eqnarray*}
\lambda_R &=&
\frac{g^\top(L - \frac{1}{2} L_\partial) g}{g^\top D Q g} \\
&= & \frac{\sum_{u \sim v} (g_u - g_v)^2 - \frac{1}{2} \sum_{\substack{ u \sim v
\\ u,v \in \partial V}} (g_u - g_v)^2}{\sum_v g_v^2 d_v q_v} \\
&\ge& \frac{\sum_{u \sim v} (g_u - g_v)^2 - \frac{1}{2} \sum_{\substack{ u \sim
v \\ u,v \in \partial V}} (g_u - g_v)^2 }{\sum_v (g(v)-g(v_p))^2d_v q_v},
\end{eqnarray*}
where the last inequality holds because we have increased the
denominator. From here,

\begin{equation} \label{eq1}
\lambda_R \geq \frac{\sum_{u \sim v} ((g^+_u - g^+_v)^2+(g^-_u-g^-_v)^2)
 - \frac{1}{2} \sum_{\substack{ u \sim v
\\ u,v \in \partial V}} ((g^+_u - g^+_v)^2+(g^-_u - g^-_v)^2)
}{\sum_v
((g^+_v)^2+(g^-_v)^2) d_v q_v},
\end{equation}


Recall that
\begin{equation} \label{eq2}
\frac{a+b}{c+d} \geq \text{min} \left\{\frac{a}{c}, \frac{b}{d} \right\},
\end{equation}

for any $a,b \geq 0$ and $c,d > 0$.  
From \eqref{eq1}, we can set $a = \sum_{u \sim v} (g^+_u - g^+_v)^2 - \sum_{\substack{ u \sim v
\\ u,v \in \partial V}} (g^+_u - g^+_v)^2,$ 
$b = \sum_{u \sim v} (g^-_u-g^-_v)^2 - \sum_{\substack{ u \sim v \\ u,v \in \partial V}} (g^-_u - g^-_v)^2,$
$c = \sum_v (g^+_v)^2 d_v q_v,$ and
$d = \sum_v (g^-_v)^2 d_v q_v.$
Observe that $a$ and $b$ are nonnegative. Indeed,
$$a = \sum_{\substack{ u \sim v \\ u \notin \partial V \ \text{or}\ v \notin \partial V}} (g^+_u - g^+_v)^2,$$
which has nonnegative summands, and a similar statement holds for $b.$

Without loss of generality, \eqref{eq2} implies that
$$
\lambda_R  \ge 
\frac{\sum_{u \sim v} (g^+_u - g^+_v)^2
 - \frac{1}{2} \sum_{\substack{ u \sim v
\\ u,v \in \partial V}} (g^+_u - g^+_v)^2
}{\sum_v
(g^+_v)^2 d_v q_v}. 
$$
To simplify notation in the following, let $f = g^+$. We begin by setting
$$
\lambda :=
\frac{\sum_{u \sim v} (f_u- f_v)^2
 - \frac{1}{2} \sum_{\substack{ u \sim v
\\ u,v \in \partial V}} (f_u- f_v)^2
}{\sum_v
f_v^2 d_v q_v}. 
$$
Multiplying the numerator and denominator by the same term gives
$$
\lambda = \frac{
\left(
\sum_{u \sim v} (f_u- f_v)^2 - \frac{1}{2} \sum_{\substack{ u \sim v \\ u,v \in
\partial V} } (f_u- f_v)^2 \right) \left(
\sum_{u \sim v} (f_u+ f_v)^2 - \frac{1}{2} \sum_{\substack{ u \sim v \\ u,v \in
\partial V} } (f_u+ f_v)^2 \right)
}{\left( \sum_v
f_v^2 d_v q_v \right) \left( \sum_{u \sim v} (f_u+ f_v)^2 - \frac{1}{2}
\sum_{\substack{ u \sim v \\ u,v \in
\partial V} } (f_u+ f_v)^2 \right)
}. 
$$
Applying the Cauchy-Schwarz inequality in the numerator gives
$$
\lambda \ge \frac{
\left(
\sum_{u \sim v} |f_u^2- f_v^2| - \frac{1}{2} \sum_{\substack{ u \sim v \\ u,v
\in \partial V} } |f_u^2- f_v^2| \right)^2 }{\left( \sum_v
f_v^2 d_v q_v \right) \left( \sum_{u \sim v} (f_u+ f_v)^2 - \frac{1}{2}
\sum_{\substack{ u \sim v \\ u,v \in \partial V} } (f_u+ f_v)^2 \right) }. 
$$
Next, we observe that
$$
\sum_{u \sim v} (f_u+ f_v)^2 - \frac{1}{2}
\sum_{\substack{ u \sim v \\ u,v \in \partial V} } (f_u+ f_v)^2  = \sum_v f_v^2
d_v q_v - \left( \sum_{u \sim v} (f_u- f_v)^2 - \frac{1}{2}
\sum_{\substack{ u \sim v \\ u,v \in \partial V} } (f_u- f_v)^2  \right),
$$
and thus it follows that
$$
\lambda \ge \frac{ \left( \sum_{u \sim v} |f_u^2- f_v^2| - \frac{1}{2}
\sum_{\substack{ u \sim v \\ u,v \in \partial V} } |f_u^2- f_v^2| \right)^2
}{\left( \sum_v f_v^2 d_v q_v \right)^2 \left(2 -\lambda \right) }. 
$$
We want to show that
$$
\sum_{u \sim v} |f_u^2- f_v^2| - \frac{1}{2} \sum_{\substack{ u \sim v \\ u,v
\in \partial V} } |f_u^2- f_v^2|  \geq \sum_{i=1}^n |f_{v_i}^2 -
f_{v_{i+1}}^2|  m(S_i,V \setminus S_i).
$$
We can write
$$
\sum_{u \sim v} |f_u^2- f_v^2| - \frac{1}{2} \sum_{\substack{ u \sim v \\ u,v
\in \partial V} } |f_u^2- f_v^2|  = \sum_{i=2}^n \sum_{j = 1}^{i-1} \left(
\chi_{E_{i,j}} - \frac{1}{2} \chi_{\partial_i} \chi_{\partial_j} \right)
(f_{v_i}^2-f_{v_j}^2),
$$
where 
$$\chi_{E_{i,j}} = \begin{cases}
1 \ \text{if}\ \{v_i, v_j\} \in E \\
0 \ \text{otherwise}
\end{cases}$$ 
is the indicator function for $\{v_i,v_j\} \in E$, and
$$\chi_{\partial_i} = \begin{cases}
1 \ \text{if}\  i \in \partial V \\
0 \ \text{otherwise}
\end{cases}$$ 
is the indicator function for $v_i \in \partial V$. Note
that we are justified in dropping the absolute value signs because $f_{v_i}^2$
is an increasing function of $i.$ Next we write $f_{v_i}^2 - f_{v_j}^2$ as a
telescoping series
$$
f_{v_i}^2 - f_{v_j}^2 = (f_{v_i}^2 - f_{v_{i-1}}^2)+(f_{v_{i-1}}^2 - f_{v_{i-2}}^2) + ... +
(f_{v_{j+1}}^2-f_{v_j}^2),
$$
and rearrange terms in the summation to conclude that
\begin{multline*}
 \sum_{i=2}^n \sum_{j = 1}^{i-1} \left( \chi_{E_{i,j}} - \frac{1}{2}
\chi_{\partial_i} \chi_{\partial_j} \right) (f_{v_i}^2-f_{v_j}^2) = \\
\sum_{l=1}^n \sum_{k=1}^n \sum_{j=1}^n \left( \left( \chi_{E_{j,k+l}} -
\frac{1}{2} \chi_{\partial_j} \chi_{\partial_{k+l}} \right)   \chi_{j\leq l}
\right) (f_{v_{l+1}}^2 - f_{v_l}^2),
\end{multline*}
where
$$\chi_{j \leq l} = \begin{cases}
1 \ \text{if}\  j \leq l \\
0 \ \text{otherwise}.
\end{cases}$$
Then, to complete this step, we note that
$$
\sum_{k=1}^n \sum_{j=1}^n \left(\left( \chi_{E_{j,k+l}} - \frac{1}{2}
\chi_{\partial_j} \chi_{\partial_{k+l}}\right) \chi_{j\leq l} \right) = m(S_l,V
\setminus S_l).
$$
Returning to our main sequence of inequalities for $\lambda$, we have
\begin{eqnarray*}
\lambda 
&\ge& \frac{(\sum_i |f_{v_i}^2 - f_{v_{i+1}}^2| m(S_i,V \setminus S_i))^2}{2(\sum_v f_v^2 d_v
q_v)^2} \\
&\geq& \frac{(\alpha \sum_{i=1}^n |f_{v_i}^2-f_{v_{i+1}}^2| \min \{
\text{vol}(S_i), \text{vol}(V\setminus S_i)\})^2}{2(\sum_u f(u)^2d_u q_v)^2},
\end{eqnarray*}
where 
$$
\alpha := \min_{1\leq i\leq n} \frac{m(S_i,V\setminus S_i)}{\min \{ \vol(S_i),
\vol(V\setminus S_i)\}} .
$$
Since $f_{v_i}^2$ is nondecreasing, a rearrangement of the numerator
of the previous expression gives
$$
\lambda \ge \frac{\alpha^2}{2} \frac{(\sum_i
(f_{v_i}^2 |\min\{\vol(S_i),\text{vol}(V\setminus S_i)\}
-\min\{\vol(S_{i+1}),\text{vol}(V \setminus S_{i+1})\}|))^2} {(\sum_u
f(u)^2d_u q_u)^2}.
$$
It follows that
$$
\lambda_R \ge \lambda \ge \frac{\alpha^2}{2} \frac{(\sum_i f_{v_i}^2d_{v_i}
q_{v_i})^2}{(\sum_u f_{u}^2 d_u q_u)^2} = \frac{\alpha^2}{2} \ge
\frac{h_R^2}{2}, 
$$
which completes the proof.
\end{proof}

\subsection*{Acknowledgements}
We thank the referees for their helpful comments. This research was supported by
Summer Undergraduate Math Research at Yale (SUMRY) 2018. NFM was supported in
part by NSF DMS-1903015.

\end{document}